\pdfoutput=1		

\documentclass[a4paper,12pt]{article}  
\usepackage[vmargin={20mm,20mm},hmargin={20mm,20mm}]{geometry}
\usepackage{commath}
\usepackage{graphicx}
\usepackage{enumitem}
\usepackage{tikz-cd}
\usepackage{subcaption}

\usepackage{amsmath}
\usepackage{amssymb}
\usepackage{amsthm}
\usepackage{bbm}
\usepackage{mathrsfs}

\DeclareMathOperator{\Aut}{Aut}

\DeclareMathOperator{\Crit}{Crit}

\DeclareMathOperator{\Diff}{Diff}

\DeclareMathOperator{\Fix}{Fix}

\DeclareMathOperator{\grad}{grad}

\DeclareMathOperator{\Ham}{Ham}
\DeclareMathOperator{\Hess}{Hess}

\DeclareMathOperator{\HM}{HM}

\DeclareMathOperator{\id}{id}
\DeclareMathOperator{\im}{im}

\DeclareMathOperator{\Int}{Int}

\DeclareMathOperator{\ord}{ord}
\DeclareMathOperator{\pr}{pr}

\DeclareMathOperator{\RFC}{RFC}
\DeclareMathOperator{\RFH}{RFH}

\DeclareMathOperator{\Symp}{Symp}

\newtheoremstyle{main} 		             	 		
  	{}	                                     		
  	{}	                                    		
  	{\itshape}			                     		
  	{}        	                             		
  	{\boldmath\bfseries}   	                         		
  	{.}            	                        		
  	{ }           	                         		
  	{\thmname{#1}\thmnumber{ #2}\thmnote{ (#3)}}	
\theoremstyle{main}
\newtheorem{definition}{Definition}[section]

\newtheorem{proposition}{Proposition}[section]
\newtheorem{corollary}{Corollary}[section]
\newtheorem{theorem}{Theorem}[section]
\newtheorem{lemma}{Lemma}[section]
\newtheoremstyle{nonit} 		             	 		
  	{}	                                     		
  	{}	                                    		
  	{}			                     		
  	{}        	                             		
	{\boldmath\bfseries}   	                         		
  	{.}            	                        		
  	{ }           	                         		
  	{\thmname{#1}\thmnumber{ #2}\thmnote{ (#3)}}	
\theoremstyle{nonit}
\newtheorem{remark}{Remark}[section]
\newtheorem{example}{Example}[section]
\newtheoremstyle{ex} 		             	 		
  	{}	                                     		
  	{}	                                    		
  	{}			                     		
  	{}        	                             		
  	{\bfseries\boldmath}   	                         		
  	{.}            	                        		
  	{ }           	                         		
  	{\thmname{#1}\thmnumber{ #2}\thmnote{ (#3)}}	
\theoremstyle{ex}

\usepackage[backend=bibtex, style=alphabetic]{biblatex}
\bibliography{bibliography}
\usepackage[bookmarksopen=true,bookmarksnumbered=true]{hyperref}
\hypersetup{
  colorlinks   = true, 
  urlcolor     = blue, 
  linkcolor    = blue, 
  citecolor    = blue 
}

\usepackage{tcolorbox}
\tcbuselibrary{theorems}
\newtcbtheorem
  []
  {summary}
  {Summary}
  {%
    colback=blue!5,
    colframe=blue!35!black,
    fonttitle=\bfseries,
  }
  {sum}

\begin{document}

\title{On a Theorem by Schlenk}



\author{Yannis B\"ahni\footnote{\href{mailto:yannis.baehni@ifv.gess.ethz.ch}{yannis.baehni@ifv.gess.ethz.ch}, MINT Learning Institute at ETH Zürich}
}

\maketitle

\begin{abstract}
	In this paper we prove a generalisation of Schlenk's theorem about the existence of contractible periodic Reeb orbits on stable, displaceable hypersurfaces in symplectically aspherical, geometrically bounded, symplectic manifolds, to a forcing result for contractible twisted periodic Reeb orbits. We make use of holomorphic curve techniques for a suitable generalisation of the Rabinowitz action functional in the stable case in order to prove the forcing result. As in Schlenk's theorem, we derive a lower bound for the displacement energy of the displaceable hypersurface in terms of the action value of such periodic orbits.  The main application is a forcing result for noncontractible periodic Reeb orbits on quotients of certain symmetric star-shaped hypersurfaces. Either there exist two geometrically distinct noncontractible periodic Reeb orbits or the period of the noncontractible periodic Reeb orbit is small. This theorem can be applied to many physical systems including the H\'enon--Heiles Hamiltonian and Stark--Zeeman systems.  Further applications include a new proof of the well-known fact that the displacement energy is a relative symplectic capacity on $\mathbb{R}^{2n}$ and that the Hofer metric is indeed a metric.
\end{abstract}

\tableofcontents

\newpage

\section{Introduction}\label{sec1}

In \cite{baehni:rfh:2023}, a generalisation of Rabinowitz--Floer homology was constructed. Rabinowitz--Floer homology is the Morse--Bott homology in the sense of Floer associated with the Rabinowitz action functional introduced by Kai Cieliebak and Urs Frauenfelder in 2009. The main application of this generalisation was to prove an existence result for noncontractible periodic Reeb orbits on quotients of certain symmetric star-shaped hypersurfaces in $\mathbb{C}^n$, $n \geq 2$. More precisely, let $\Sigma \subseteq \mathbb{C}^n$ be a compact and connected star-shaped hypersurface invariant under the rotation	
\begin{equation*}
	\varphi \colon \mathbb{C}^n \to \mathbb{C}^n, \quad \varphi(z_1,\dots,z_n) := \left( e^{2\pi i k_1/m}z_1,\dots,e^{2\pi i k_n/m}z_n\right)
\end{equation*}
\noindent for some even $m \geq 2$ and $k_1,\dots,k_n \in \mathbb{Z}$ coprime to $m$. Then $\Sigma/\mathbb{Z}_m$ admits a noncontractible periodic Reeb orbit generating the fundamental group $\pi_1(\mathbb{S}^{2n - 1}/\mathbb{Z}_m) \cong \mathbb{Z}_m$. For a proof see \cite[Theorem~1.2]{baehni:rfh:2023} and \cite[Theorem~1.1]{abreu:symmetric:2022} for the more general result, removing the restriction of $m$ being even. The existence of noncontractible periodic Reeb
orbits on lens spaces is extremely relevant and attracts much attention in celestial mechanics as mentioned in \cite[Introduction]{liu:RP:2022} or \cite{pedro:reeb:2016}. We quickly recall the setup for the proof of this result. Let $(W,\lambda)$ be a connected Liouville domain with connected boundary $\partial W$ and consider a Liouville automorphism $\varphi \in \Aut(W,\lambda)$, that is, $\varphi \in \Diff(W)$ is of finite order and there exists a unique function $f_\varphi \in C^\infty(\Int W)$ such that $\varphi^*\lambda - \lambda = df_\varphi$. The main step was to construct a homology theory for the \emph{twisted Rabinowitz action functional}
\begin{equation*}
    \mathscr{A}^H_\varphi \colon \mathscr{L}_\varphi M \times \mathbb{R} \to \mathbb{R}, \qquad \mathscr{A}^H_\varphi(\gamma,\tau) := \int_0^1 \gamma^*\lambda - \tau \int_0^1 H(\gamma(t))dt - f_\varphi(\gamma(0))
\end{equation*}
\noindent on the completion $(M,\lambda)$ of $(W,\lambda)$, where
\begin{equation*}
    \mathscr{L}_\varphi M := \{\gamma \in C^\infty(\mathbb{R},M) : \gamma(t + 1) = \varphi(\gamma(t))\> \forall t \in \mathbb{R}\}
\end{equation*}
\noindent denotes the \emph{twisted loop space of $M$ and $\varphi$}. Twisted loops play a significant role in physical systems with symmetries, see for example \cite[Section~6.2]{cieliebak:frozen:2023}. Consider the chain complex $\RFC^\varphi(\partial W,M)$ generated by a suitable Morse function on the critical manifold $\Crit(\mathscr{A}^H_\varphi)$, where
\begin{equation*}
    (\gamma,\tau) \in \Crit(\mathscr{A}^H_\varphi) \qquad \Leftrightarrow \qquad \begin{cases}
        \gamma \in \mathscr{L}_\varphi \partial W,\\
        \dot{\gamma}(t) = \tau R(\gamma(t)) \> \forall t \in \mathbb{R},
    \end{cases}
\end{equation*}
\noindent with $R \in \mathfrak{X}(\partial W)$ denoting the Reeb vector field. We then define twisted Rabinowitz-Floer homology as the Morse–Bott homology with coefficients in $\mathbb{Z}_2$  by
    \begin{equation*}
        \RFH^\varphi(\partial W,M) := \HM(\mathscr{A}^H_\varphi) = \frac{\ker \partial \colon \RFC^\varphi(\partial W, M) \to \RFC^\varphi(\partial W, M)}{\im \partial \colon \RFC^\varphi(\partial W, M) \to \RFC^\varphi(\partial W, M)},
    \end{equation*}
    \noindent where the boundary map $\partial$ counts twisted negative gradient flow lines modulo two with respect to a suitable $d\lambda$-compatible almost complex structure on $M$. This homology theory has the following crucial properties.
\begin{enumerate}
	\item The semi-infinite dimensional Morse--Bott homology $\RFH^\varphi(\partial W, M)$ is well-defined. Moreover, twisted Rabinowitz--Floer homology is invariant under twisted homotopies of Liouville domains.
    \item Twisted Rabinowitz--Floer homology is indeed a generalisation of the standard Rabinowitz--Floer homology $\RFH(\partial W,M)$ defined in \cite{cieliebakfrauenfelder:rfh:2009}, as
    \begin{equation*}
        \RFH^{\id_W}(\partial W, M) \cong \RFH(\partial W,M).
    \end{equation*}
	\item If $\partial W$ is simply connected and does not admit any nonconstant twisted periodic Reeb orbits, then
    \begin{equation*}
           \RFH^\varphi_*(\partial W,M) \cong \operatorname{H}_*(\Fix(\varphi\vert_{\partial W});\mathbb{Z}_2).
        \end{equation*}
        Note that $\Fix(\varphi)$ is a symplectic submanifold of $M$ by \cite[Lemma~5.5.7]{mcduffsalamon:st:2017}.
	\item If $\partial W$ is displaceable by a compactly supported Hamiltonian symplectomorphism in the completion $(M,\lambda)$, then
        \begin{equation*}
            \RFH^\varphi(\partial W,M) \cong 0.
        \end{equation*}
	\end{enumerate}
	For a proof see \cite[Theorem~1.1]{baehni:rfh:2023}. Note that there are two possible ways for proving property 4: either one shows that the norm of the gradient of a perturbed version of the twisted Rabinowitz action functional is uniformly bounded from below as in \cite[Lemma~3.9]{cieliebakfrauenfelder:rfh:2009}, or one generalises leaf-wise intersection points following \cite{albersfrauenfelder:rfh:2010}. A direct consequence of properties 3 and 4 is the following observation as in \cite[Corollary~1.5]{cieliebakfrauenfelder:rfh:2009}. Suppose that $\partial W$ is Hamiltonianly displaceable in the completion $(M,\lambda)$ and simply connected. If $\Fix(\varphi\vert_{\partial W}) \neq \varnothing$, then $\partial W$ does admit a twisted periodic Reeb orbit. Indeed, if there does not exist any twisted periodic Reeb orbit on the boundary $\partial W$, we compute using property 3
\begin{equation*}
	\RFH^\varphi(\partial W,M) \cong \operatorname{H}(\Fix(\varphi\vert_{\partial W});\mathbb{Z}_2) = \bigoplus_{j \geq 0} \operatorname{H}_j(\Fix(\varphi\vert_{\partial W});\mathbb{Z}_2) \neq 0,
\end{equation*}
\noindent contradicting property 4. However, if $\Fix(\varphi\vert_{\partial W}) = \varnothing$, then one cannot directly conclude the existence of a twisted periodic Reeb orbit on $\partial W$. This is for example the case for the rotation $\varphi \colon \mathbb{C}^n \to \mathbb{C}^n$ from the beginning. So the best one can hope for is some kind of forcing result to hold. More precisely, if we know that there exists a sufficiently well-behaved twisted periodic Reeb orbit, then this \emph{forces} the existence of another one. The above observation is already a forcing result, as $\Fix(\varphi\vert_{\partial W})$ is precisely the set of all constant twisted periodic Reeb orbits on the contact manifold $\partial W$.

\newpage

\section{Results}\label{sec2}
\subsection{Preliminaries on Twisted Stable Hypersurfaces}
\begin{definition}[{Stable Hypersurface, \cite[p.~1774]{cieliebakfrauenfelderpaternain:mane:2010}}]
	Let $(M,\omega)$ be a connected symplectic manifold. A \emph{stable hypersurface} in $(M,\omega)$ is a compact and connected hypersurface $\Sigma \subseteq M$ such that the following conditions hold:
	\begin{enumerate}
		\item $\Sigma$ is separating, that is, $M \setminus \Sigma$ consists of two connected components $M^\pm$, where $M^-$ is bounded and $M^+$ is unbounded.
		\item There exists a vector field $X$ in a neighbourhood of $\Sigma$ such that $X$ is outward-pointing to $\Sigma \cup M^-$ and $\ker \omega\vert_\Sigma \subseteq \ker L_X\omega\vert_\Sigma$.
	\end{enumerate}
 We write $(\Sigma,\omega\vert_\Sigma,\lambda)$ for a stable hypersurface, where the stabilising form $\lambda \in \Omega^1(\Sigma)$ is defined by $\lambda := i_X\omega\vert_\Sigma$.
\end{definition}

\begin{definition}[Twisted Stable Hypersurface]
	\label{def:twisted_stable_hypersurface}
	Let $(\Sigma,\omega\vert_\Sigma, \lambda)$ be a stable hypersurface in a connected symplectic manifold $(M,\omega)$ and $\varphi \in \Symp(M,\omega)$. We say that $\Sigma$ is \emph{twisted} by $\varphi$, if $\varphi(\Sigma) = \Sigma$, $\varphi$ is of finite order and $\varphi^*X = X$.
\end{definition}

\begin{example}[Star-Shaped Hypersurfaces]
    \label{ex:star-shaped}
    Consider the Liouville automorphism
    \begin{equation*}
        \varphi \colon \mathbb{C}^n \to \mathbb{C}^n, \qquad \varphi(z_1,\dots,z_n) := \left(e^{2\pi i k_1/m}z_1,\dots, e^{2\pi ik_n/m}z_n\right)
    \end{equation*}
    \noindent for $m \geq 2$ an integer and $k_1,\dots,k_n \in \mathbb{Z}$ coprime to $m$. Let $f \in C^\infty(\mathbb{S}^{2n - 1})$ be a positive function such that $f  \circ \varphi = f$. Then the star-shaped hypersurface
    \begin{equation*}
        \Sigma_f = \{f(z)z : z \in \mathbb{S}^{2n - 1}\} \subseteq \mathbb{C}^n
    \end{equation*}
    \noindent is a contact manifold with $\varphi$-invariant contact form $\lambda\vert_{\Sigma_f}$, where
    \begin{equation*}
        \lambda := \frac{1}{2}\sum_{j = 1}^n(y_j dx_j - x_j dy_j) = \frac{i}{4}\sum_{j = 1}^n (\overline{z}_j dz_j - z_j d\overline{z}_j)
    \end{equation*}
    \noindent with complex coordinates $z_j  = x_j + iy_j$. Indeed, by \cite[Lemma~12.2.2]{frauenfelderkoert:3bp:2018}, we have that
    \begin{equation*}
        X_{H_f}\vert_{\Sigma_f} \in \ker d\lambda\vert_{\Sigma_f} \qquad \text{and} \qquad \lambda(X_{H_f})\vert_{\Sigma_f} = 1
    \end{equation*}
    \noindent for the defining Hamiltonian function
    \begin{equation*}
        H_f \colon \mathbb{C}^n \setminus \{0\} \to \mathbb{R}, \qquad H_f(z) := \frac{\|z\|^2}{f(z/\|z\|^2)} - 1.
    \end{equation*}
    Hence $(\Sigma_f,\lambda\vert_{\Sigma_f})$ is a contact manifold as the Liouville vector field
    \begin{equation*}
        X := \frac{1}{2}\sum_{j = 1}^n\left(x_j\frac{\partial}{\partial x_j} + y_j \frac{\partial}{\partial y_j}\right) \in \mathfrak{X}(\mathbb{R}^{2n})
    \end{equation*}
    \noindent satisfies $i_X d\lambda = \lambda$ and is outward-pointing as
    \begin{equation*}
        dH_f(X)\vert_{\Sigma_f} = d\lambda(X,X_{H_f})\vert_{\Sigma_f} = \lambda(X_{H_f})\vert_{\Sigma_f} = 1.
    \end{equation*}
    Finally, we conclude that
    \begin{equation*}
        X_{H_f}\vert_{\Sigma_f} = R_f \in \mathfrak{X}(\Sigma_f)
    \end{equation*}
    \noindent is the Reeb vector field. 
\end{example}

\begin{example}[{Magnetic Torus, \cite[Section~6.1]{cieliebakfrauenfelderpaternain:mane:2010}}]
	\label{ex:magnetic_torus}
	Let $\mathbb{T}^n$ be the standard flat torus for $n \geq 2$ and let $J \colon \mathbb{R}^n \to \mathbb{R}^n$ be  an antisymmetric nonzero linear map. Define $\rho \in \Omega^2(\mathbb{T}^n)$ by setting $\rho(\cdot,\cdot) := \langle \cdot, J \cdot \rangle$ and denote by $\omega_\rho = dp \wedge dq + \pi^*\rho$ the magnetic symplectic form on $T^*\mathbb{T}^n \cong \mathbb{T}^n \times \mathbb{R}^n$. For an energy value $k \in \mathbb{R}$ set $\Sigma_k := H^{-1}(k)$ for the mechanical Hamiltonian function 
	\begin{equation*}
		H(q,p) := \frac{1}{2} \|p\|^2 \qquad \forall (q,p) \in \mathbb{T}^n \times \mathbb{R}^n.
	\end{equation*}
	Define $A := (J\vert_{\im J})^{-1}$ and $\alpha \in \Omega^1(\im J)$ by
	\begin{equation*}
		\alpha_x(v) := \frac{1}{2}\langle x,Av\rangle.
	\end{equation*}
	Then $\Sigma_k$ is a displaceable stable Hamiltonian manifold for every $k > 0$ by \cite[Proposition~6.3]{cieliebakfrauenfelderpaternain:mane:2010}. The stabilising form $\lambda$ on $\Sigma_k$ is given by
	\begin{equation}
		\label{eq:stabilising_form}
		\lambda = f^*(pdq) + (\pr_\parallel \circ \pr)^*\alpha,
	\end{equation}
	\noindent where 
	\begin{equation*}
		\pr_\perp \colon \mathbb{R}^n \to \ker J, \qquad \pr_\parallel \colon \mathbb{R}^n \to \im J, \qquad \pr \colon \mathbb{T}^n \times \mathbb{R}^n \to \mathbb{R}^n
	\end{equation*}
	\noindent denote the projections with respect to the orthogonal splitting
	\begin{equation*}
		\mathbb{R}^n = \ker J \oplus \im J,
	\end{equation*}
	\noindent and
	\begin{equation*}
		f \colon \mathbb{T}^n \times \mathbb{R}^n \to \mathbb{T}^n \times \mathbb{R}^n, \qquad f(q,p) := \left(q,\pr_\perp(p)\right).
	\end{equation*}
    Let $\varphi \in \Diff(\mathbb{T}^n)$ be an isometry of finite order satisfying 
	\begin{equation}
		\label{eq:twist_condition}
		D\varphi \circ J = J \circ D\varphi
	\end{equation}
	\noindent and consider the cotangent lift
	\begin{equation*}
		D\varphi^\dagger \colon \mathbb{T}^n \times \mathbb{R}^n \to \mathbb{T}^n \times \mathbb{R}^n, \qquad D\varphi^\dagger(q,p) = \left(\varphi(q),\left(D\varphi^{-1}(q)\right)^T p\right).
	\end{equation*}
	Then clearly $\varphi(\Sigma_k) = \Sigma_k$ as $\varphi$ is an isometry and $D\varphi^\dagger$ is of finite order as $\varphi$ is. Moreover, we have that $D\varphi^\dagger \in \Symp(T^*\mathbb{T}^n,\omega_\rho)$, as $D\varphi^\dagger \in \Symp(T^*\mathbb{T}^n,dp \wedge dq)$ and $D\varphi^\dagger$ preserves $\rho$ by assumption \eqref{eq:twist_condition}. Lastly, we have that $\varphi^*\lambda = \lambda$ by considering formula \eqref{eq:stabilising_form} together with assumption \eqref{eq:twist_condition}, and thus also $\varphi^*X = X$ by \cite[Proposition~4.2]{abbondandolo:lagrangian:2013}. 
\end{example}

\begin{definition}[{Hofer Norm, \cite[p.~466]{mcduffsalamon:st:2017}}]
	Let $(M,\omega)$ be a symplectic manifold. Define the \emph{Hofer norm} of $F \in C^\infty_c(M \times [0,1])$ by
	\begin{equation*}
		\|F\| := \|F\|_+ + \|F\|_-,
	\end{equation*}
	\noindent where
	\begin{equation*}
		\|F\|_+ := \int_0^1 \max_{x \in M} F_t(x)dt \qquad \text{and} \qquad \|F\|_- := -\int_0^1 \min_{x \in M} F_t(x)dt. 
	\end{equation*}
\end{definition}

\begin{definition}[{Displacement Energy, \cite[p.~469]{mcduffsalamon:st:2017}}]
    \label{def:e}
	Let $(M,\omega)$ be a symplectic manifold. For a compact subset $A \subseteq M$ define the \emph{displacement energy} of $A$ by
	\begin{equation*}
		e(A) := \inf_{\substack{F \in C^\infty_c(M \times [0,1])\\\varphi_F(A) \cap A = \varnothing}}\|F\|,
	\end{equation*}
    \noindent where $\varphi_F := \phi^{X_F}_1$ denotes the time-1-map of the smooth flow of the time-dependent Hamiltonian vector field $X_{F_t}$.
\end{definition}

\begin{example}[{\cite[p.~189]{paternain:e:2008}}]
    Let $M$ be a compact manifold without boundary. Then we have that $e(M) = +\infty$ in $(T^*M,dp \wedge dq)$ for the zero-section $M$ in $T^* M$. However, if $\rho \neq 0$ for a magnetic cotangent bundle $(T^*M, \omega_\rho)$ and $\chi(M) = 0$ for the Euler-characteristic $\chi$ of $M$, then $e(M) < +\infty$ is finite. For more examples of nondisplaceable hypersurfaces in cotangent bundles see \cite[Theorem~1.13]{merry:rfh:2014}.
\end{example}

\begin{definition}[{Symplectic Asphericity, \cite[p.~302]{mcduffsalamon:J-holomorphic_curves:2012}}]
	A connected symplectic manifold $(M,\omega)$ is said to be \emph{symplectically aspherical}, if
	\begin{equation*}
		\int_{\mathbb{S}^2} f^*\omega = 0 \qquad \forall f \in C^\infty(\mathbb{S}^2,M).
	\end{equation*}
 Equivalently, $(M,\omega)$ is symplectically aspherical if and only if for the de-Rham-homology class $[\omega]\vert_{\pi_2(M)} = 0 \in \mathrm{H}_{\mathrm{dR}}^2(M;\mathbb{R})$.
\end{definition}

\begin{example}[Magnetic Torus]
    \label{ex:aspherical}
    The magnetic torus $(T^*\mathbb{T}^n,\omega_\rho)$ from Example \ref{ex:magnetic_torus} is symplectically aspherical as $\omega_\rho = d\lambda_\theta$ is exact with
    \begin{equation*}
        \lambda_\theta := pdq + \pi^*\theta, \qquad \theta_q(\cdot) = \frac{1}{2}\langle q, J\cdot\rangle \in T^*_q\mathbb{T}^n
    \end{equation*}
    \noindent for all $q \in \mathbb{T}^n$ by \cite[Lemma~6.2]{cieliebakfrauenfelderpaternain:mane:2010}, where $\pi \colon T^*\mathbb{T}^n \to \mathbb{T}^n$ denotes the projection. Alternatively, the magnetic cotangent bundle $(T^*\mathbb{T}^n,\omega_\rho)$ is symplectically aspherical as we have $\pi_2(T^*\mathbb{T}^n) \cong \pi_2(\mathbb{T}^n) \times \pi_2(\mathbb{R}^n) = 0$.
\end{example}

\begin{definition}[Contractible Twisted Loop Space]
    Let $(M,\omega)$ be a symplectic manifold and $\varphi \in \Symp(M,\omega)$ of finite order. A loop $v \in C^\infty(\mathbb{T},M)$, $\mathbb{T} := \mathbb{R}/\mathbb{Z}$, is said to be a \emph{contractible twisted periodic loop}, if there exists $\gamma \in \mathscr{L}_\varphi M$ such that
    \begin{equation*}
        v(t) = \gamma(\ord(\varphi)t) \qquad \forall t \in \mathbb{T},
    \end{equation*}
    \noindent and a filling $\bar{v} \in C^\infty(\mathbb{D},M)$ on the unit disc
    \begin{equation*}
        \mathbb{D} := \{z \in \mathbb{C} : |z| = 1\},
    \end{equation*}
    \noindent such that $\bar{v}(e^{2\pi i t}) = v(t)$ for all $t \in \mathbb{T}$. We denote the space of all contractible twisted periodic loops of $M$ and $\varphi$ by $\Lambda_\varphi M$.
\end{definition}

\begin{definition}[Twisted Rabinowitz Action Functional]
    Let $(\Sigma,\omega\vert_\Sigma,\lambda)$ be a twisted stable hypersurface in a symplectically aspherical symplectic manifold $(M,\omega)$. For a defining Hamiltonian function $H$ for $\Sigma$ with $H \circ \varphi = H$, we define the \emph{twisted Rabinowitz action functional}
    \begin{equation*}
        \mathscr{A}^H_\varphi \colon \Lambda_\varphi M \times \mathbb{R} \to \mathbb{R}, \qquad \mathscr{A}^H_\varphi(v,\tau) := \frac{1}{\ord(\varphi)}\int_{\mathbb{D}} \bar{v}^*\omega - \tau \int_0^1 H(v(t))dt.
    \end{equation*}
\end{definition}

\begin{remark}[$\Crit(\mathscr{A}^H_\varphi)$]
    Let $X \in \mathfrak{X}(\gamma)$ be a twisted variation, that is, $X$ is a vector field along $\gamma$ and satisfies the condition
    \begin{equation*}
        X(t + 1) = D\varphi(X(t)) \qquad \forall t \in \mathbb{R}.
    \end{equation*}
    Then a routine computation shows that 
    \begin{equation*}
        (v,\tau) \in \Crit(\mathscr{A}^H_\varphi) \qquad \Leftrightarrow \qquad \begin{cases}
            \gamma \in \mathscr{L}_\varphi \Sigma,\\
            \dot{\gamma}(t) = \tau R(\gamma(t)) \> \forall t \in [0,1],
        \end{cases}
    \end{equation*}
    \noindent where $R \in \mathfrak{X}(\Sigma)$ is the stable Reeb vector field. If $J$ is a $\varphi$-invariant almost complex structure compatible with $\omega$, then the gradient $\grad_J \mathscr{A}^H_\varphi \in \mathfrak{X}(\Lambda_\varphi M \times \mathbb{R})$ with respect to the $L^2$-metric
    \begin{equation*}
        m\left((X,\eta),(Y,\sigma)\right) := \int_0^1 \omega(JX(t),Y(t))dt + \eta\sigma \qquad \forall (X,\eta), (Y,\sigma) \in T_{(v,\tau)} \Lambda_\varphi M \times \mathbb{R},
    \end{equation*}
    \noindent and $(v,\tau) \in \Lambda_\varphi M \times \mathbb{R}$, is given by
    \begin{equation*}
        \grad_J \mathscr{A}^H_\varphi\vert_{(v,\tau)}(t) = \begin{pmatrix}
            J(\dot{\gamma}(\ord(\varphi)t) - \tau X_H(v(t)))\\
            \displaystyle - \int_0^1 H \circ v
        \end{pmatrix} \qquad \forall t \in \mathbb{T}.
    \end{equation*}
    Hence $(u,\tau) \in C^\infty(\mathbb{R},\Lambda_\varphi M \times \mathbb{R})$ is a twisted negative gradient flow line, if the elliptic partial differential equations or \emph{twisted Rabinowitz--Floer equations}
    \begin{equation*}
        \partial_s u(s,t) + J\left(\partial_t \gamma(s,\ord(\varphi)t) - \tau(s) X_H(u(s,t))\right) = 0 \quad \text{and} \quad  \partial_s \tau(s) = \int_0^1 H(u(s,t))dt
    \end{equation*}
    \noindent hold for all $(s, t) \in \mathbb{R} \times \mathbb{T}$.
\end{remark}

\begin{example}[Magnetic Torus]
	\label{ex:twisted_Reeb}
	Consider the displaceable twisted stable hypersurface $\Sigma_k \subseteq (T^*\mathbb{T}^n,\omega_\rho,H)$ as in Example \ref{ex:magnetic_torus}. A point $(q,p) \in \Sigma_k$ gives rise to a twisted periodic Reeb orbit if and only if
	\begin{equation*}
		\int_0^\tau e^{sJ}p ds + q = \varphi(q), \quad e^{\tau J}p = \left(D\varphi^{-1}(q)\right)^Tp, \quad \text{and} \quad \|p\|^2 = 2k.
	\end{equation*}
	A computation similar to \cite[p.~1843]{cieliebakfrauenfelderpaternain:mane:2010} shows
	\begin{equation*}
		\mathscr{A}^H_\varphi(v,\tau) = \ord(\varphi)k\tau.
	\end{equation*}
\end{example}

\begin{definition}[{Morse--Bott Component, \cite[p.~86]{albersfrauenfelder:rfh:2010}}]
    \label{def:mb_component}
	Let $\mathscr{A} \colon \mathscr{E} \to \mathbb{R}$ be a smooth functional. A subset $C \subseteq \Crit \mathscr{A}$ is called a \emph{Morse--Bott component}, if
	\begin{enumerate}
		\item $C$ is an action-constant submanifold of $\mathscr{E}$.
		\item $T_xC = \ker \Hess \mathscr{A}(x)$ for all $x \in C$ for the Hessian $\Hess \mathscr{A}$ of $\mathscr{A}$.
	\end{enumerate}
\end{definition}

\begin{example}[$\Fix(\varphi\vert_\Sigma)$]
    \label{ex:fix}
    Let $\Sigma$ be a twisted stable hypersurface in a symplectically aspherical symplectic manifold. Then $\Fix(\varphi\vert_\Sigma) \subseteq \Crit \mathscr{A}^H_\varphi$ is a Morse--Bott component. Indeed, by \cite[Proposition~2.23]{baehni:rfh:2023} we have that
	\begin{equation*}
		\ker \Hess \mathscr{A}^H_\varphi\vert_{(x,0)} \cong \ker (D\varphi_x - \id_{T_x \Sigma}) = T_x\Fix(\varphi\vert_\Sigma)
	\end{equation*}
	\noindent for all $x \in \Fix(\varphi\vert_\Sigma)$.
\end{example}

\begin{definition}[{\cite[p.~1768]{cieliebakfrauenfelderpaternain:mane:2010}}]
    A symplectic manifold $(M,\omega)$ is called \emph{geometrically bounded}, if there exists an $\omega$-compatible almost complex structure $J$ and a complete Riemannian metric such that the following conditions hold.
    \begin{enumerate}
		\item There are constants $C_0,C_1 > 0$ with
        \begin{equation*}
            \omega(Jv,v) \geq C_0\|v\|^2 \qquad \text{and} \qquad |\omega(u,v)| \leq C_1\|u\|\|v\|
        \end{equation*}
        \noindent for all $u,v \in T_x M$ and $x \in M$.
        \item The sectional curvature of the metric is bounded above, and its injectivity radius is bounded away from zero.
    \end{enumerate}
\end{definition}

\begin{example}[{\cite[p.~1768]{cieliebakfrauenfelderpaternain:mane:2010}}]
	\label{ex:geometrically_bounded}
	Magnetic cotangent bundles are geometrically bounded.
\end{example}

\subsection{A Forcing Theorem for Twisted Periodic Reeb Orbits}
Let $(W,\lambda)$ be a connected Liouville domain with connected boundary $\Sigma := \partial W$. Let $(M,\lambda)$ be the completion of $(W,\lambda)$ and $\varphi \in \Aut(W,\lambda)$ a Liouville automorphism, that is, $\varphi \in \Diff(W)$ is a diffeomorphism of finite order such that $\varphi^*\lambda = \lambda$. In this setup, the kernel of the twisted Rabinowitz action functional $\mathscr{A}^H_\varphi$ admits the canonical description
\begin{equation*}
    \ker \Hess \mathscr{A}^H_\varphi\vert_{(v,\tau)} \cong \ker \left(D^\xi(\phi_{-\tau}^R \circ \varphi)\vert_{v(0)} - \id\vert_{\xi_{v(0)}}\right) \oplus \langle R(v(0)) \rangle
\end{equation*}
\noindent by \cite[Proposition~2.23]{baehni:rfh:2023}, where $\xi := \ker \lambda\vert_\Sigma$ denotes the contact distribution.

\begin{definition}[{Transversal Nondegeneracy, \cite[Definition~7.3.1]{frauenfelderkoert:3bp:2018}}]
    Let $(M,\lambda)$ be the completion of a connected Liouville domain $(W,\lambda)$. A contractible twisted periodic Reeb orbit $(v,\tau) \in \Crit(\mathscr{A}_\varphi^H)$ is said to be \emph{nondegenerate}, if
    \begin{equation*}
        \ker \left(D^\xi(\phi_{-\tau}^R \circ \varphi)\vert_{v(0)} - \id\vert_{\xi_{v(0)}}\right) = \{0\}.
    \end{equation*}
\end{definition}

\begin{theorem}
	\label{thm:multiplicity}
	Let $\Sigma \subseteq \mathbb{C}^n$, $n \geq 2$, be a compact and connected star-shaped hypersurface invariant under the rotation	
	\begin{equation*}
		\varphi \colon \mathbb{C}^n \to \mathbb{C}^n, \quad \varphi(z_1,\dots,z_n) := \left(e^{2\pi i k_1/m}z_1,\dots,e^{2\pi i k_n/m}z_n\right)
	\end{equation*}
	\noindent for some $m \geq 2$ and $k_1,\dots,k_n \in \mathbb{Z}$ coprime to $m$. Assume that there exists a nondegenerate noncontractible simple periodic Reeb orbit $(\gamma_0,\tau_0)$ on $\Sigma/\mathbb{Z}_m$. Then there does exist a noncontractible periodic Reeb orbit $(\gamma,\tau)$ on $\Sigma/\mathbb{Z}_m$ such that
		\begin{equation*}
		0 < \tau - \tau_0 \leq e(\Sigma).
	\end{equation*}	
    Consequently, we have two cases.
    \begin{enumerate}
        \item If $\gamma$ is not an interate of $\gamma_0$, then $\Sigma/\mathbb{Z}_m$ admits two geometrically distinct noncontractible periodic Reeb orbits.
        \item If $\gamma$ is a $p$-fold iterate of $\gamma_0$, then the period $\tau_0$ is small in the sense of 
        \begin{equation*}
            0 < \tau_0 \leq \frac{1}{p - 1}e(\Sigma).
        \end{equation*}
    \end{enumerate}
\end{theorem}

\begin{remark}
    We cannot conclude the existence of two geometrically distinct noncontractible periodic Reeb orbits as in \cite[Theorem~1.2]{abreu:symmetric:2022} from Theorem \ref{thm:multiplicity} even under the additional assumption that $\Sigma$ is dynamically convex. Indeed, the sphere $\Sigma = \mathbb{S}^{2n - 1}$ is dynamically convex by the Hofer--Wysocki--Zehnder Theorem \cite[Theorem~12.2.1]{frauenfelderkoert:3bp:2018}, and both cases do occur there.
\end{remark}

\begin{example}[{The H\'enon--Heiles Hamiltonian, \cite[Section~2]{schneider:lens:2020}}]
    \label{ex:henon}
    Consider the mechanical Hamiltonian function
    \begin{equation*}
        H \colon \mathbb{R}^4 \to \mathbb{R}, \qquad H(q_1,q_2,p_1,p_2) := \frac{1}{2}\left(\|p\|^2 + \|q\|^2\right) + q_1^2 q_2 - \frac{1}{3}q_2^3. 
    \end{equation*}
    This Hamiltonian function is known as the \emph{H\' enon--Heiles Hamiltonian}. On $\mathbb{R}^4 \cong \mathbb{C}^2$ consider the coordinates
    \begin{equation*}
        z := q_1 + iq_2 \qquad \text{and} \qquad w := p_1 + i p_2.
    \end{equation*}
    Define
    \begin{equation*}
        \varphi \colon \mathbb{C}^2 \to \mathbb{C}^2, \qquad \varphi(z,w) := e^{2\pi i/3}(z,w).
    \end{equation*}
    We have that $\varphi^*\lambda = \lambda$ for
    \begin{equation*}
        \lambda = \frac{1}{2}(p_1 dq_1 - q_1 dp_1) + \frac{1}{2}(p_2 dq_2 - q_2 dp_2).
    \end{equation*}
    For every $0 < k < \frac{1}{6}$, the regular energy surface $H^{-1}(k)$ contains a strictly convex sphere-like component $\Sigma_k \cong \mathbb{S}^3$. The resulting quotient $\Sigma_k/\mathbb{Z}_3$ is diffeomorphic to the lens space $L(3,1)$, but not contactomorphic to it with the standard contact distribution. Here we write $L(m,k_2)$ for the lens space $\mathbb{S}^3/\mathbb{Z}_m$ from Example \ref{ex:star-shaped} with $k_1 = 1$. Instead, the quotient $\Sigma_k/\mathbb{Z}_3$ is contactomorphic to $L(3,2)$ with its standard contact distribution. This is mainly due to the use of different coordinates. By a shooting argument, one can show that there exist at least two $\mathbb{Z}_3$-symmetric periodic orbits on $\Sigma_k$. In fact, by \cite[Corollary~2.5]{schneider:lens:2020}, there exist infinitely many periodic orbits on $\Sigma_k$.
\end{example}

\begin{example}[{Hill's Lunar Problem, \cite[Section~5.8]{frauenfelderkoert:3bp:2018}}]
    \label{ex:lunar}
    The mechanical Hamiltonian function $H \colon T^*(\mathbb{R}^2 \setminus \{0\}) \to \mathbb{R}$ defined by
    \begin{equation*}
        H(q,p) := \frac{1}{2}\| p\|^2 - \frac{1}{\|q\|} + q_1p_2 - q_2p_1 - q_1^2 + \frac{1}{2}q_2^2 
    \end{equation*}
    \noindent is called \emph{Hill's lunar Hamiltonian}. After Levi--Civita regularisation the regularised Hill's lunar Hamiltonian $K \colon T^* \mathbb{R}^2 \to \mathbb{R}$ is given by
    \begin{equation*}
        K(q,p) = \frac{1}{2}\left(\| p \|^2 + \| q \|^2 \right) + 2\|q\|^2(q_2 p_1 - q_1 p_2) - 4(q_1^6 - 3q_1^4 q_2^2 - 3q_1^2 q_2^4 + q_2^6).
    \end{equation*}
    For $k > 0$ sufficiently small, the energy hypersurface $K^{-1}(k)$ admits at least two periodic orbits by \cite[Theorem~1]{roberto:hill:2011} and contains a strictly convex sphere-like component $\Sigma_k \cong \mathbb{S}^3$. On $T^*\mathbb{R}^2 \cong \mathbb{C}^2$ consider the coordinates
    \begin{equation*}
        z := q_1 + iq_2 \qquad \text{and} \qquad w := p_1 + i p_2
    \end{equation*}
    \noindent and the rotation
    \begin{equation*}
        \varphi \colon \mathbb{C}^2 \to \mathbb{C}^2, \qquad \varphi(z,w) := e^{\pi i/2}(z,w).
    \end{equation*}
    Then $K$ is invariant under the rotation $\varphi$ and thus $\Sigma/\mathbb{Z}_4$ is diffeomorphic to the lens space $L(4,1)$, but again due to the choice of nonstandard coordinates not contactomorphic to it. It is a delicate question in Contact Topology to decide the correct value of $k_2 \neq 1$, such that the obtained lens space $L(4,1)$ in Hill's lunar problem is contactomorphic to $L(4,k_2)$. The tight contact
structures on the lens spaces $L(m,k_2)$ are classified up to isotopy by \cite[Therorem~2.1]{honda:lens:2000}, so in principle it should be possible to obtain the correct value of $k_2$.
\end{example}

\begin{example}[Stark--Zeeman Systems]
    \label{ex:stark}
    Planar Stark--Zeeman systems as in \cite{cieliebak:J:2017} and \cite{cieliebak:two:2023} generalise many important physical systems including the diamagnetic Kepler problem and the restricted three body problem \cite{moreno:3bp:2022}. By \cite[Corollary~1]{cieliebak:J:2017}, for energy values below the first critical value, the Moser regularised energy hypersurfaces are diffeomorphic to the unit cotangent bundles $S^*\mathbb{S}^n$. In particular, for $n = 2$ we obtain $S^*\mathbb{S}^2 \cong \mathbb{RP}^3$,  a real projective space. 
\end{example}

Theorem \ref{thm:multiplicity} immediately follows from a more general result.

\begin{theorem}[Forcing]
	\label{thm:forcing}
	Let $\Sigma$ be a twisted stable displaceable hypersurface in a symplectically aspherical, geometrically bounded, symplectic manifold $(M,\omega)$ for a symplectomorphism $\varphi \in \Symp(M,\omega)$ of finite order $\ord(\varphi)$ and suppose that $v_0$ is a contractible twisted periodic Reeb orbit on $\Sigma$ belonging to a Morse--Bott component $C$. Then there exists a contractible twisted periodic Reeb orbit $v \notin C$ such that
	\begin{equation*}
		\int_{\mathbb{D}}\bar{v}^*\omega - \int_{\mathbb{D}} \bar{v}_0^*\omega \leq \ord(\varphi)e(\Sigma).
	\end{equation*}	
\end{theorem}

\begin{remark}
    The case $(M,\Sigma) = (\mathbb{C}^n, \mathbb{S}^{2n - 1})$ with the rotation
    \begin{equation*}
	   \varphi \colon \mathbb{C}^n \to \mathbb{C}^n, \quad \varphi(z_1,\dots,z_n) = \left( e^{2\pi i k_1/m}z_1,\dots,e^{2\pi i k_n/m}z_n\right),
    \end{equation*}
    \noindent shows that the estimate in Theorem \ref{thm:forcing} is sharp.
\end{remark}

Applying Theorem \ref{thm:forcing} to the Morse--Bott component $\Fix(\varphi\vert_\Sigma)$ from Example \ref{ex:fix} yields the following corollary.

\begin{corollary}
    \label{cor:schlenk}
    Let $\Sigma$ be a twisted stable displaceable hypersurface in a symplectically aspherical, geometrically bounded, symplectic manifold $(M,\omega)$ for $\varphi \in \Symp(M,\omega)$ with $\Fix(\varphi\vert_\Sigma) \neq \varnothing$. Then there does exist a nonconstant contractible twisted periodic Reeb orbit $v$ such that
	\begin{equation*}
		\int_{\mathbb{D}}\bar{v}^*\omega \leq \ord(\varphi)e(\Sigma).
	\end{equation*}	
\end{corollary}

In particular, if we take $\varphi = \id_M$ in Corollary \ref{cor:schlenk}, we recover Schlenk's theorem as stated in \cite[Theorem~4.9]{cieliebakfrauenfelderpaternain:mane:2010} about the existence of contractible closed characteristics on stable, displaceable hypersurfaces with energy less or equal to the displacement energy of the hypersurface. Schlenk proved this result in \cite[Theorem~1.1]{schlenk:stable:2006} using quite different methods.

\begin{example}[Magnetic Torus]
    \label{ex:magnetic_forcing}
    We can apply Theorem \ref{thm:forcing} and its Corollary \ref{cor:schlenk} to the magnetic torus in Example \ref{ex:magnetic_torus}. Indeed, $(T^*\mathbb{T}^n,\omega_\rho)$ is geometrically bounded by Example \ref{ex:geometrically_bounded} and symplectically aspherical by Example \ref{ex:aspherical}. Moreover, $\Sigma_k$ is stable and displaceable for every energy value $k > 0$. Thus for every contractible twisted periodic Reeb orbit $v_0$ belonging to a Morse--Bott component, there does exist a contractible twisted periodic Reeb orbit $v$ with
    \begin{equation*}
        \int_0^1 v^* \lambda_\theta - \int_0^1 v_0^*\lambda_\theta \leq \ord(\varphi)e(\Sigma).
    \end{equation*}

\end{example}

Further applications of Theorem \ref{thm:forcing} and its Corollary \ref{cor:schlenk} are the content of the next section. The proof of Theorem \ref{thm:forcing} is given in Section \ref{sec:proof}. It is also the aim of future research to numerically investigate the Examples \ref{ex:henon}, \ref{ex:lunar} and \ref{ex:stark}, that is, finding upper bounds of the displacement energy and minimal periods.

\newpage

\section{Applications}
\subsection{The Hofer Distance and Relative Symplectic Capacities}
Computing the displacement energy \ref{def:e} is usually quite difficult. Sometimes it is possible to give upper bounds on the displacement energy as in \cite[Theorem~1]{irie:e:2014} or lower bounds as for any nonempty open subset $A \subseteq M$ of a symplectic manifold $(M,\omega)$ we have $e(A) > 0$ as in \cite[Theorem~1.1]{banyaga:e:2018}. Corollary \ref{cor:schlenk} has two immediate consequences. First, the existence of a nonconstant contractible twisted periodic Reeb orbit on any twisted stable displaceable hypersurface. Second, the existence of a lower bound for the displacement energy via the action value of this critical point. If the hypersurface is of contact type, this action value is precisely the period of the parametrised periodic Reeb orbit. We illustrate the usefulness of the second implication and give dynamical proofs of standard results. Recall, that a \emph{relative symplectic capacity on $\mathbb{R}^{2n}$} is a map $c$ which assigns to each subset $A \subseteq \mathbb{R}^{2n}$ a number $c(A) \in [0,+\infty]$ such that the following three properties hold \cite[p.~460]{mcduffsalamon:st:2017}.
\begin{enumerate}
    \item (\emph{Relative Monotonicity}) If there exists a symplectomorphism $\psi$ of $\mathbb{R}^{2n}$ such that $\psi(A) \subseteq B$, then $c(A) \leq c(B)$.
    \item (\emph{Conformality}) $c(\lambda A) = \lambda^2 c(A)$ for all $\lambda \in \mathbb{R}$.
    \item (\emph{Normalisation}) It holds that
    \begin{equation*}
        c(B^{2n}(r)) = c(Z^{2n}(r)) = \pi r^2 \qquad \forall r > 0,
    \end{equation*}
    \noindent for the closed ball of radius $r$
    \begin{equation*}
        B^{2n}(r) := \left\{(x,y) \in \mathbb{R}^{2n} : \|x\|^2 + \|y\|^2 \leq r^2\right\},
    \end{equation*}
    \noindent and the closed cylinder
    \begin{equation*}
        Z^{2n}(r) := \left\{(x,y) \in \mathbb{R}^{2n} : x_1^2 + y_1^2 \leq r^2\right\}.
    \end{equation*}
\end{enumerate}

\begin{proposition}[{\cite[Theorem~12.3.4]{mcduffsalamon:st:2017}}]
    \label{prop:e}
	The displacement energy $e$ is a relative symplectic capacity on $\mathbb{R}^{2n}$. 
\end{proposition}

\begin{proof}
    Relative monotonicity and conformality are not hard to show. Moreover, by relative monotonicity and \cite[Exercise~12.3.7]{mcduffsalamon:st:2017} we have that
	\begin{equation*}
		e(\partial B^{2n}(r)) \leq e(B^{2n}(r)) \leq e(Z^{2n}(r)) \leq \pi r^2 \qquad \forall r > 0.
	\end{equation*}
	The periodic Reeb flow on $\partial B^{2n}(r)$ is given by
	\begin{equation*}
		\phi^{R_r}_t(z) = e^{-2i t/r^2}z \qquad \forall z \in \partial B^{2n}(r).
	\end{equation*}
	Hence the parametrised periodic Reeb orbits are $(\phi^{R_r}(z),\tau)$ with $\tau \in \pi r^2 \mathbb{Z}$. But Corollary \ref{cor:schlenk} implies the existence of a nonconstant closed periodic Reeb orbit $(v,\tau)$ on the contact hypersurface $\partial B^{2n}(r)$ such that
	\begin{equation*}
		0 < \tau = \int_0^1 v^*\lambda \leq e(\partial B^{2n}(r)) \leq \pi r^2,
	\end{equation*}
    \noindent where
    \begin{equation*}
        \lambda := \frac{1}{2}\sum_{j = 1}^n (y_j dx_j - x_j dy_j).
    \end{equation*}
	This is only possible for $\tau = \pi r^2$ and the statement follows.
\end{proof}

\begin{proposition}[{\cite[Theorem~1.1]{banyaga:e:2018}}]
    \label{prop:nontriviality}
    For any subset $A \subseteq \mathbb{R}^{2n}$ with nonempty interior we have that $e(A) > 0$.
\end{proposition}

\begin{proof}
    If $A \subseteq M$ is not displaceable, we have that $e(A) = +\infty$ and thus there is nothing to show. Moreover, if $A$ is not compact, we define
    \begin{equation*}
        e(A) := \sup_{K \subseteq A \text{ $K$ compact}} e(K).
    \end{equation*}
    So we can assume that $A$ is displaceable by a compactly supported Hamiltonian symplectomorphism $\varphi_F \in \Ham_c(\mathbb{R}^{2n},dy \wedge dx)$. As $A$ is displaceable and has nonempty interior, there exists a closed ball $B(r)$ of radius $r$ such that
    \begin{equation*}
        \varphi_F(B(r)) \cap B(r) = \varnothing.
    \end{equation*}
    Since the displacement energy is a normalised relative symplectic capacity by Proposition \ref{prop:e}, we conclude that
    \begin{equation*}
        e(A) \geq e(B(r)) = \pi r^2 > 0.
    \end{equation*}
\end{proof}

\begin{corollary}[{Hofer Distance, \cite[Theorem~12.3.3]{mcduffsalamon:st:2017}}]
    On $\Ham_c(\mathbb{R}^{2n}, dy \wedge dx)$ define the \emph{Hofer distance}
    \begin{equation*}
		\rho(\varphi_0,\varphi_1) := \inf_{\varphi_F = \varphi_1 \circ \varphi_0^{-1}}\|F\|.
	\end{equation*}
    Then
    \begin{equation*}
        \rho(\varphi_0,\varphi_1) = 0 \quad \Rightarrow \quad \varphi_0 = \varphi_1 \qquad \forall \varphi_0,\varphi_1 \in \Ham_c(M,\omega),
    \end{equation*}
    \noindent that is, the Hofer distance is a metric on $\Ham_c(\mathbb{R}^{2n}, dy \wedge dx)$.
\end{corollary}

\begin{proof}
    Let $\varphi \in \Ham_c(\mathbb{R}^{2n}, dy \wedge dx)$ be not equal to the identity. Thus there exists a set $A$ with nonempty interior such that $\varphi(A) \cap A = \varnothing$. Thus Lemma \ref{prop:nontriviality} implies
    \begin{equation*}
        \rho(\varphi,\id_{\mathbb{R}^{2n}}) \geq e(A) > 0.
    \end{equation*}
    This proves the statement.
\end{proof}

\begin{remark}
    In \cite[Corollary~1.2]{mcduff:e:1995}, these results are generalised to arbitrary symplectic manifolds.
\end{remark}

\subsection{Physical Systems and the Ma\~n\'e Critical Value}

\begin{proposition}
    \label{prop:tonelli}
    Let $(T^*M,dp \wedge dq, H)$ be a Hamiltonian system for a compact configuration space $M$ and define
    \begin{equation*}
        e_0(H) := \inf\left\{k \in \mathbb{R} : \pi_{T^*M}\left(H^{-1}(k)\right) = M\right\},
    \end{equation*}
    \noindent where $\pi_{T^*M} \colon T^*M \to M$ denotes the projection. Suppose that $\Sigma_k := H^{-1}(k)$ with $k < e_0(H)$ is a $\varphi$-twisted stable regular energy surface admitting a contractible twisted periodic Reeb orbit $(q_0,p_0)$ belonging to a Morse--Bott component $C$. Then there exists a contractible twisted periodic Reeb orbit $(q,p) \notin C$ such that
    \begin{equation*}
        \int_0^1 p(t)\dot{q}(t)dt - \int_0^1 p_0(t)\dot{q}_0(t)dt \leq \ord(\varphi) e(\Sigma_c).
    \end{equation*}
\end{proposition}

\begin{proof}
    We claim that $e(\Sigma_k) < +\infty$ for all $k < e_0(H)$. In particular, every energy hypersurface $\Sigma_k$ is displaceable in the geometrically bounded and symplectically aspherical symplectic manifold $(T^*M, dp \wedge dq)$ since $T^*M$ is an exact symplectic form with canonical Liouville form $pdq$. As $k < e_0(H)$, we can displace $\Sigma_k$ into the missing fibres. The explicit compactly supported Hamiltonian symplectomorphism achieving that is constructed in \cite[Proposition~8.2]{contreras:periodic:2006}. Hence if $\Sigma_k$ is twisted stable and $k < e_0(H)$, we conclude the existence of such a contractible periodic Reeb orbit from Theorem \ref{thm:forcing}.
\end{proof}

\begin{example}[Magnetic Torus]
    Let $M$ be a compact manifold and $\theta \in \Omega^1(M)$. Then the map
    \begin{equation*}
        \varphi_\theta \colon (T^*M, dp \wedge dq) \to (T^*M, \omega_{d\theta}), \qquad \varphi_\theta(q,p) := (q, p - \theta_q)
    \end{equation*}
    \noindent is an exact symplectomorphism. Indeed, for every $(q,p) \in T^*M$ and $v \in TT^*_{(q,p)}M$ we compute
    \begin{align*}
        (\varphi^*_{-\theta} \lambda)_{(q,p)}(v) &= \lambda_{(q,p + \theta_q)}\left(D\varphi_{-\theta}(v)\right)\\
        &= p\left(D\pi_{T^*M}(v)\right) + \theta_q\left(D\pi_{T^*M}(v)\right)\\
        &= (\lambda + \pi_{T^*M}^*\theta)_{(q,p)}(v),
    \end{align*}
    \noindent where $\lambda \in \Omega^1(T^*M)$ denotes the canonical Liouville form and $\varphi_{-\theta} \circ \varphi_\theta = \id_{T^*M}$. A mechanical Hamiltonian function
    \begin{equation*}
        H \colon (T^*M, \omega_{d\theta}) \to \mathbb{R}, \qquad H(q,p) = \frac{1}{2}\|p\|_{m^*}^2 + V(q),
    \end{equation*}
    \noindent for some potential function $V \in C^\infty(M)$ is transformed under $\varphi_\theta$ to a magnetic Hamiltonian function $H_\theta = \varphi_\theta^* H$ given by
    \begin{equation*}
        H_\theta \colon (T^*M, dp \wedge dq) \to \mathbb{R}, \qquad H_\theta(q,p) = \frac{1}{2}\|p - \theta_q\|^2_{m^*} + V(q).
    \end{equation*}
    In the case of a magnetic torus as in Example \ref{ex:magnetic_forcing}, we have that
    \begin{equation*}
        \theta_q(v) = \frac{1}{2}\langle q, Jv\rangle \qquad \forall (q,v) \in \mathbb{T}^n \times \mathbb{R}^n.
    \end{equation*}
    Thus if $k > 0$, the intersection of $\Sigma_k = H^{-1}_\theta(k)$ with $T_q^*\mathbb{T}^n$ is a sphere centred at $\theta_q$ for every $q \in \mathbb{T}^n$. For more details see \cite[Example~5.2]{abbondandolo:floer:2019}. Consequently, we have that $e_0 = 0$ and Proposition \ref{prop:tonelli} cannot be applied. Note that the Ma\~n\'e critical value $c$ is infinite in this case because a nonzero $\rho$ has no bounded primitives in $\mathbb{R}^n$.
\end{example}

\begin{remark}
    In the setting of Proposition \ref{prop:tonelli}, if $H$ is a Tonelli Hamiltonian function, that is, $H$ is strictly fibrewise convex and superlinear, then any stable energy level of $H$ does contain a periodic Reeb orbit by \cite{macarini:stable:2010}. See also \cite[Theorem~(iv)]{abbondandolo:lagrangian:2013}.
\end{remark}

\begin{remark}
    The proof of Proposition \ref{prop:tonelli} does not work for higher energy values in general. This is due to a theorem of Will Merry in \cite[Theorem~1.1]{merry:rfh:2011} and \cite[Remark~1.7]{merry:rfh:2011}. Let $H \in C^\infty(T^*M)$ be a Tonelli Hamiltonian function. Define the Ma\~n\'e critical value
    \begin{equation*}
        c := \inf_{\theta}\sup_{q \in \widetilde{M}}\widetilde{H}(q,\theta_q),
    \end{equation*}
    \noindent where the infimum is taken over all $1$-forms $\theta$ on the universal covering manifold $\widetilde{M}$ with $d\theta = \widetilde{\rho}$, and $\widetilde{H} \in C^\infty(T^*\widetilde{M})$ denotes the lift of $H$. We always have that
    \begin{equation*}
        c \geq e_0(H).
    \end{equation*}
    If $k > c$, then the Rabninowitz--Floer homology $\RFH_*(\Sigma_k,T^*M)$ of \cite{cieliebakfrauenfelderpaternain:mane:2010} is well-defined and does not vanish. In particular, $\Sigma_k$ is not displaceable. Thus we cannot apply Theorem \ref{thm:forcing} in that case.
\end{remark}

\section{Proof of Theorem \ref{thm:forcing}}
\label{sec:proof}

The proof of Theorem \ref{thm:forcing} uses a method called a ``homotopy of homotopies argument''. Fix $\varepsilon > 0$ and choose a Hamiltonian function $F \in C^\infty_c(M \times [0,1])$ satisfying 
\begin{equation*}
    \textstyle F_t = 0 \> \forall t \in [0,\frac{1}{2}], \qquad \|F\| < e(\Sigma) + \varepsilon \qquad \text{and} \qquad \varphi_F(\Sigma) \cap \Sigma = \varnothing.
\end{equation*}
This is possible by definition of the displacement energy \ref{def:e}. Next we need to carefully choose a twisted defining Hamiltonian function $H$ for the stable hypersurface $\Sigma$. We postpone the construction of this Hamiltonian function and explain the main idea of the proof. Choose a smooth family $(\beta_r)_{r \in [0,+\infty)}$ of cutoff functions $\beta_r \in C^\infty(\mathbb{R},[0,1])$ such that
\begin{equation*}
	\begin{cases}
		\beta_r(s) = 0 & |s| \geq r,\\
		\beta_r(s) = 1 & |s| \leq r - 1,\\
		s\beta'_r(s) \leq 0 & \forall s \in \mathbb{R},
	\end{cases}
\end{equation*}
\noindent for all $r \in [0,+\infty)$. Define a family of twisted Rabinowitz action functionals
\begin{equation*}
	\mathscr{A}_r \colon \Lambda_\varphi M \times \mathbb{R} \times \mathbb{R} \to \mathbb{R}
\end{equation*}
\noindent by
\begin{equation*}
	\mathscr{A}_r(v,\tau,s) := \mathscr{A}^H_\varphi(v,\tau) - \beta_r(s)\int_0^1 F_t(v(t))dt
\end{equation*}
\noindent for all $r \in [0,+\infty)$. Note that $\mathscr{A}_0 = \mathscr{A}^H_\varphi$. For a suitable $\varphi$-invariant $\omega$-compatible almost complex structure we consider the moduli space
\begin{equation*}
	\mathscr{M} := \{(u,\tau,r) \in C^\infty(\mathbb{R}, \mathscr{L}_\varphi M \times \mathbb{R}) \times [0,+\infty) : (u,\tau,r) \text{ solution of \eqref{eq:moduli_space}}\},
\end{equation*}
\noindent where
\begin{equation}
	\begin{cases}
	    \partial_s (u,\tau) = \grad \mathscr{A}_r\vert_{(u(s),\tau(s),s)} & \forall s \in \mathbb{R},\\
	    \displaystyle\lim_{s \to -\infty} (u(s),\tau(s)) = (v_0,\tau_0),\\
	    \displaystyle\lim_{s \to +\infty} (u(s),\tau(s)) \in C.
	\end{cases}
	\label{eq:moduli_space}
\end{equation}
Note that always $(v_0,\tau_0,0) \in \mathscr{M}$ and that such a $\varphi$-invariant $\omega$-compatible almost complex structure always exists by \cite[Lemma~5.5.6]{mcduffsalamon:st:2017}. The gradient $\grad \mathscr{A}_r$ of $\mathscr{A}_r$ is taken with respect to the metric
\begin{equation*}
    m\left((X,\eta),(Y,\sigma)\right) := \int_0^1 \omega(JX(t),Y(t))dt + \eta\sigma.
\end{equation*}

\begin{lemma}
    \label{lem:compact}
    If 
    \begin{equation}
        \label{eq:assumption_moduli_space}
        \mathscr{A}_0(v,\tau) > \|F\| + \mathscr{A}_0(v_0,\tau_0) \qquad \forall (v,\tau) \in \Crit(\mathscr{A}_0) \setminus C,
    \end{equation}
    \noindent then $\mathscr{M}$ is compact.
\end{lemma}

As a corollary of Lemma \ref{lem:compact} we get Theorem \ref{thm:forcing}. Indeed,  the moduli space $\mathscr{M}$ is the zero level set of a Fredholm section of a bundle over a Banach manifold. As $v_0$ belongs to a Morse--Bott component, the Fredholm section is regular at the point $v_0$, that is, the linearisation of the gradient flow equation is surjective there. By compactness, we can therefore perturb the Fredholm section to make it transverse. Hence $\mathscr{M}$ is a compact smooth manifold with boundary consisting precisely of the point $v_0$. See \cite[Appendix~A]{mcduffsalamon:J-holomorphic_curves:2012} for details. This is absurd, and we conclude that there exists a critical point $(v,\tau) \in \Crit(\mathscr{A}_0) \setminus C$ such that 
\begin{equation*}
	\mathscr{A}_0(v,\tau) - \mathscr{A}_0(v_0,\tau_0) \leq \|F\| < e(\Sigma) + \varepsilon.
\end{equation*}
As $\varepsilon > 0$ was arbitrary, the statement follows since
\begin{equation*}
    \mathscr{A}_0(v,\tau) = \frac{1}{\ord(\varphi)}\int_{\mathbb{D}} \bar{v}^*\omega \qquad \forall (v,\tau) \in \Crit(\mathscr{A}_0).
\end{equation*}

We prove Lemma \ref{lem:compact} in four steps.

\noindent \emph{Step 1: If $(u,\tau,r) \in \mathscr{M}$, then $E(u,\tau) \leq \|F\|$ for the energy}
\begin{equation*}
    E(u,\tau) := \int_{-\infty}^{+\infty} \|\partial_s (u,\tau)\|^2_Jds.
\end{equation*}
We estimate
\allowdisplaybreaks
\begin{align*}
	E(u,\tau) &= \int_{-\infty}^{+\infty} \|\partial_s (u,\tau)\|_J^2ds\\
	&= \int_{-\infty}^{+\infty} d\mathscr{A}_r(\partial_s (u,\tau),s) ds\\
	&= \int_{-\infty}^{+\infty} \frac{d}{ds}\mathscr{A}_r(u,\tau,s) ds - \int_{-\infty}^{+\infty} (\partial_s \mathscr{A}_r)(u,\tau,s)ds\\
	&= \lim_{s \to +\infty} \mathscr{A}_r(u,\tau,s) - \lim_{s \to -\infty} \mathscr{A}_r(u,\tau,s) - \int_{-\infty}^{+\infty} (\partial_s \mathscr{A}_r)(u,\tau,s)ds\\
	&= \mathscr{A}_0(v,\tau) - \mathscr{A}_0(v_0,\tau_0) - \int_{-\infty}^{+\infty} (\partial_s \mathscr{A}_r)(u,\tau,s)ds\\
	&= -\int_{-\infty}^{+\infty} (\partial_s \mathscr{A}_r)(u,\tau,s)ds\\
	&= \int_{-\infty}^{+\infty} \dot{\beta}_r(s)\int_0^1 F_t(u(s,t))dtds\\
	&\leq \|F\|_+ \int_{-\infty}^0 \dot{\beta}_r(s)ds - \|F\|_-\int_0^{+\infty}\dot{\beta}_r(s)ds\\
	&= \beta_r(0)(\|F\|_- + \|F\|_+)\\
	&= \beta_r(0)\|F\|\\
	&\leq \|F\|,
\end{align*}
\noindent as $\mathscr{A}_0(v,\tau) = \mathscr{A}_0(v_0,\tau_0)$ since $C$ is action-constant by definiton of a Morse--Bott component \ref{def:mb_component}. 
    
\noindent \emph{Step 2: There exists $r_0 \in \mathbb{R}$ such that $r \leq r_0$ for all $(u,\tau,r) \in \mathscr{M}$.} Crucial is the existence of a constant $\delta > 0$ such that
	\begin{equation*}
		\|\grad \mathscr{A}_r\vert_{(v,\tau,s)}\|_J \geq \delta \qquad \forall (v,\tau,s) \in \Lambda_\varphi M \times \mathbb{R} \times \mathbb{R}.
	\end{equation*}
	This is proven along the lines of \cite[Lemma~3.9]{cieliebakfrauenfelder:rfh:2009}. With this inequality and Step 1 we estimate
	\begin{equation*}
		\|F\| \geq E(u,\tau) \geq \int_{-r}^r \|\grad \mathscr{A}_r\vert_{(u(s),\tau(s),s)}\|^2_J ds \geq 2r\delta^2,
	\end{equation*}
	\noindent and thus we can set
	\begin{equation*}
		r_0 := \frac{\|F\|}{2\delta^2}.
	\end{equation*}

\noindent \emph{Step 3: There exists a constant $C > 0$ such that $\|\tau\|_\infty \leq C$ for all $(u,\tau,r) \in \mathscr{M}$.} This is a delicate estimate based on the explicit construction of the defining Hamiltonian $H$ for $\Sigma$ as well as an extension of the stabilising form $\lambda$. The bound on the Lagrange multiplier is derived by comparing the twisted Rabinowitz action functional to a different action functional. This modified version of the twisted Rabinowitz action functional is obtained using a suitable extension of the $\varphi$-invariant stabilising form $\lambda \in \Omega^1(\Sigma)$ to a compactly supported form $\beta_\lambda \in \Omega^1(M)$. The precise constructions can be found in \cite[Section~4.2.2]{cieliebakfrauenfelderpaternain:mane:2010}. Given $\beta_\lambda$, we can define the auxiliary action functional
\begin{equation*}
    \label{eq:auxiliary}
    \widehat{\mathscr{A}}_0\colon \Lambda_\varphi M \times \mathbb{R} \to\mathbb{R},\qquad \widehat{\mathscr{A}}_0(v,\tau) := \int_0^1 v^*\beta_\lambda - \tau \int_0^1 H(v(t))dt.
\end{equation*}
Moreover, we consider the bilinear form on the tangent bundle $T\Lambda_\varphi M \times \mathbb{R}$
\begin{equation*}
    \widehat{m}\left((X,\eta),(Y,\sigma)\right) := \int_0^1 d\beta_\lambda(JX(t),Y(t))dt + \eta\sigma.
\end{equation*}
The main point in the choice of the $\varphi$-invariant $H \in C^\infty(M)$, $\beta_\lambda \in \Omega^1(M)$ and the $\omega$-compatible $\varphi$-invariant almost complex structure $J$ is to make sure, that the properties 
\begin{enumerate}
    \item $d\widehat{\mathscr{A}}_0(v,\tau)(X,\eta) = \widehat{m}\left(\grad \mathscr{A}_0(v,\tau),(X,\eta)\right)$,
    \item $(m - \widehat{m})\left((X,\eta),(X,\eta)\right) \leq 0$,
\end{enumerate}
\noindent are true for all $(v,\tau) \in \Lambda_\varphi M \times \mathbb{R}$ and $(X,\eta) \in T_{(v,\tau)}\Lambda_\varphi M \times \mathbb{R}$. These two conditions ensure that the difference $\mathscr{A}_0 - \widehat{\mathscr{A}}_0$ is a Liapunov function for the negative gradient flow lines of the twisted Rabinowitz action functional $\mathscr{A}_0$. The uniform bound on the Lagrange multiplier $\tau$ now follows from Steps 1 and 2. For details see \cite[p.~1808]{cieliebakfrauenfelderpaternain:mane:2010}. The only subtle difference in our case is, that everything needs to be $\varphi$-invariant. However, this is no problem as we explain now. The construction of $H$, $\beta_\lambda$ and $J$ is based on the existence of a stable tubular neighbourhood of $\Sigma$, that is, a pair $(\rho_0,\psi)$ with $\rho_0 > 0$ and $\psi \colon (-\rho_0,\rho_0) \times \Sigma \hookrightarrow M$ an embedding such that
\begin{equation*}
    \psi\vert_{\{0\} \times \Sigma} = \iota_\Sigma \colon \Sigma \hookrightarrow M\qquad \text{and} \qquad \psi^*_\rho \omega = \omega\vert_\Sigma + d(\rho \lambda).
\end{equation*}
By \cite[Proposition~2.6~(a)]{cieliebakfrauenfelderpaternain:mane:2010}, the space of stable tubular neighbourhoods of $(\Sigma,\lambda)$ is nonempty. Using the quivariant Darboux--Weinstein Theorem \cite[Theorem~22.1]{guilleminsternberg:sg:1984}, we get the existence of a stable tubular neighbourhood $(\rho_0,\psi)$, satisfying
\begin{equation}
    \label{eq:twist}
    \varphi(\psi(\rho,x)) = \psi(\rho,\varphi(x)) \qquad \forall (\rho,x) \in (-\rho_0,\rho_0)\times \Sigma.
\end{equation}
Compare also \cite[Equation~(3.2)]{baehni:rfh:2023}. Hence the constructions \cite[p.~1791--1793]{cieliebakfrauenfelderpaternain:mane:2010} yield $\varphi$-invariant data $H$, $\beta_\lambda$ and $J$ due to \eqref{eq:twist}.

\noindent \emph{Step 4: Proof of Lemma \ref{lem:compact}.} Let $(u_k,\tau_k,r_k)$ be a sequence in the moduli space $\mathscr{M}$. By Step 2 and Step 3, the sequences $(r_k)$ and $(\tau_k)$ are uniformly bounded. Thus $(u_k,\tau_k,r_k)$ admits a $C^\infty_{\mathrm{loc}}$-convergent subsequence by standard arguments \cite[Theorem~B.4.2]{mcduffsalamon:J-holomorphic_curves:2012}. Indeed, the uniform $L^\infty$-bound on the sequence $(u_k)$ follows from the assumption that $(M,\omega)$ is geometrically bounded and the uniform $L^\infty$-bound on the derivatives $(Du_k)$ follows from the absence of bubbling as $(M,\omega)$ is symplectically aspherical. In particular, there cannot exist a nonconstant $J$-holomorphic sphere when the sequence of derivatives is unbounded \cite[Section~4.2]{mcduffsalamon:J-holomorphic_curves:2012}. Denote the limit of this subsequence by $(u,\tau,r)$. This limit clearly satisfies the first equation in \eqref{eq:moduli_space}, thus one only needs to check the asymptotic conditions in \eqref{eq:moduli_space}. Again by compactness, $(u,\tau)$ converges to critical points $(w_\pm,\tau_\pm)$ of $\mathscr{A}_0$ at its asymptotic ends. We claim that
	\begin{equation}
		\label{eq:action_value}
	    \mathscr{A}_r(u(s),\tau(s),s) \in \left[-\|F\| + \mathscr{A}_0(v_0,\tau_0), \|F\| + \mathscr{A}_0(v_0,\tau_0)\right] \qquad \forall s \in \mathbb{R}.
	\end{equation}
	In particular, $\mathscr{A}_0(w_\pm,\tau_\pm) \in \left[-\|F\| + \mathscr{A}_0(v_0,\tau_0), \|F\| + \mathscr{A}_0(v_0,\tau_0)\right]$. So if \eqref{eq:action_value} holds, then by assumption \eqref{eq:assumption_moduli_space} we conclude $(w_\pm,\tau_\pm) \in C$ and $\mathscr{M}$ is indeed compact. It remains to prove \eqref{eq:action_value}. It is enough to show
	\begin{equation*}
	    \mathscr{A}_r(u_k(s),\tau_k(s),s) \in \left[-\|F\| + \mathscr{A}_0(v_0,\tau_0), \|F\| + \mathscr{A}_0(v_0,\tau_0)\right] \qquad \forall s \in \mathbb{R}
	\end{equation*}
	\noindent for every $k \in \mathbb{N}$. As in the proof of \cite[Lemma~2.8]{albersfrauenfelder:rfh:2010} we estimate 
	\allowdisplaybreaks
	\begin{align*}
	    0 &\leq \int^{+\infty}_{s_0} d\mathscr{A}_r(\partial_s (u_k,\tau_k),s) ds\\
		&= \int^{+\infty}_{s_0} \frac{d}{ds}\mathscr{A}_r(u_k,\tau_k,s) ds - \int^{+\infty}_{s_0} (\partial_s \mathscr{A}_r)(u_k,\tau_k,s)ds\\
		&= \lim_{s \to +\infty} \mathscr{A}_r(u_k,\tau_k,s) - \mathscr{A}_r(u_k(s_0),\tau_k(s_0),s_0) - \int^{+\infty}_{s_0} (\partial_s \mathscr{A}_r)(u_k,\tau_k,s)ds\\
		&= \mathscr{A}_0(v,\tau) - \mathscr{A}_r(u_k(s_0),\tau_k(s_0),s_0) + \int^{+\infty}_{s_0} \dot{\beta}_r(s)\int_0^1 F_t(u_k(s,t))dtds\\ 
		&\leq \mathscr{A}_0(v,\tau) - \mathscr{A}_r(u_k(s_0),\tau_k(s_0),s_0) + \int^{+\infty}_{-\infty} \|\dot{\beta}_r(s)F\|_+ds\\
		&\leq \mathscr{A}_0(v,\tau) - \mathscr{A}_r(u_k(s_0),\tau_k(s_0),s_0) + \|F\|\\
		&= \mathscr{A}_0(v_0,\tau_0) - \mathscr{A}_r(u_k(s_0),\tau_k(s_0),s_0) + \|F\|
	\end{align*}
	\noindent for all $s_0 \in \mathbb{R}$. Similarly, we compute
	\allowdisplaybreaks
	\begin{align*}
	    0 &\leq \int_{-\infty}^{s_0} d\mathscr{A}_r(\partial_s (u_k,\tau_k),s) ds\\
		&= \int_{-\infty}^{s_0} \frac{d}{ds}\mathscr{A}_r(u_k,\tau_k,s) ds - \int_{-\infty}^{s_0} (\partial_s \mathscr{A}_r)(u_k,\tau_k,s)ds\\
		&= \mathscr{A}_r(u_k(s_0),\tau_k(s_0),s_0) - \lim_{s \to -\infty} \mathscr{A}_r(u_k,\tau_k,s) - \int_{-\infty}^{s_0} (\partial_s \mathscr{A}_r)(u_k,\tau_k,s)ds\\
		&= \mathscr{A}_r(u_k(s_0),\tau_k(s_0),s_0) - \mathscr{A}_0(v_0,\tau_0) + \int_{-\infty}^{s_0} \dot{\beta}_r(s)\int_0^1 F_t(u_k(s,t))dtds,
	\end{align*}
	\noindent and thus we estimate
	\begin{align*}
	    \mathscr{A}_r(u_k(s_0),\tau_k(s_0),s_0) &\geq \mathscr{A}_0(v_0,\tau_0) - \int_{-\infty}^{s_0} \dot{\beta}_r(s)\int_0^1 F_t(u_k(s,t))dtds\\
	    &\geq \mathscr{A}_0(v_0,\tau_0) - \int_{-\infty}^{+\infty} \|\dot{\beta}(s)F\|_+ds\\
	    &\geq \mathscr{A}_0(v_0,\tau_0) - \|F\|.
	\end{align*}
    This shows the estimate \eqref{eq:action_value} and so the proof of Lemma \ref{lem:compact} and Theorem \ref{thm:forcing} is complete. 

\section*{Acknowledgements}
To Colin, Neil and Jil. In memory of Will J.
Merry. A brilliant teacher and a guiding light. Without him I would have never met Urs Frauenfelder, Kai Cieliebak and Felix Schlenk.
\newpage
\addcontentsline{toc}{section}{References}
\printbibliography

\end{document}